\documentclass[12pt]{amsart}

\usepackage{newcent,euler}

\usepackage[pagebackref,colorlinks,linkcolor=red,citecolor=blue,urlcolor=blue,hypertexnames=true]{hyperref}
\usepackage{amsmath, amssymb, amsthm, latexsym,nicefrac}
\usepackage[alphabetic]{amsrefs}

\theoremstyle{plain}
\newtheorem{theorem}{Theorem}[section]
\newtheorem{question}[theorem]{Question}
\newtheorem{lemma}[theorem]{Lemma}
\newtheorem{corollary}[theorem]{Corollary}
\newtheorem{proposition}[theorem]{Proposition}
\newtheorem{conjecture}[theorem]{Conjecture}

\theoremstyle{definition}
\newtheorem{definition}[theorem]{Definition}

\theoremstyle{remark}
\newtheorem{remark}[theorem]{Remark}
\newtheorem{example}[theorem]{Example}

\input{xy}
\xyoption{all}

\def\mn{\par\medskip\noindent}

\def\Zz{\mathbb Z}
\def\Qz{\mathbb Q}
\def\Cz{\mathbb{C}}
\def\Rz{\mathbb{R}}

\def\Nz{\mathbb{N}}

\def\disc{{\rm disc}}

\def\supp{{\rm supp}}

\def\deq{\stackrel{\rm def}{=}}

\def\gal{{\rm Gal}(\overline{\Qz}\colon \Qz)}
\def\cap{{\rm cap}}

\title{Integer operators in finite von Neumann algebras}
\author{Andreas Thom}
\address{Andreas Thom, Mathematisches Institut der Universit\"at G\"ottingen,
Bunsenstr. 3-5, D-37073 G\"ottingen, Germany}
\email{thom@uni-math.gwdg.de}
\urladdr{http://www.uni-math.gwdg.de/thom}

\begin{document}

\begin{abstract}
Motivated by the study of spectral properties of self-adjoint operators in the 
integral group ring of a sofic group, we define and study integer operators.
We establish a relation with classical potential theory and in particular the circle of results
obtained by M.\ Fekete and G.\ Szeg\"o, see \cites{fek,feksze,sze}. 
More concretely, we use results by R.\ Rumely, see \cite{rumely}, on 
equidistribution of algebraic integers to obtain a description of those integer operator which 
have spectrum of logarithmic capacity less or equal to one.

Finally, we relate the study of integer operators to a recent 
construction by B.\ and L.\ Petracovici and A.\ Zaharescu, see \cite{za1}.
\end{abstract}

\maketitle

\tableofcontents

\section{Introduction}

Let $\Gamma$ be a discrete group. We consider the Hilbert space $\ell^2\Gamma$ with basis 
$\{\delta_{\gamma} \mid \gamma \in \Gamma\}$.
In \cite{thom2}, we studied spectral properties of operators which arise as integral linear combinations
of left translation operators $\lambda(\gamma) \in B(\ell^2 \Gamma)$, for $\gamma \in \Gamma$ where
$$\lambda(\gamma)\left( \sum_{\eta \in \Gamma} \xi_{\eta}  \delta_{\eta} \right) = \sum_{\eta \in \Gamma} \xi_{\eta} 
\delta_{\gamma \eta}.$$ 
We identify these left convolution operators with the integral group ring $\Zz \Gamma$. All these operators commute with the similarly defined right convolution operators $\rho(\gamma)$, for $\gamma \in \Gamma$:
$$\rho(\gamma) \left( \sum_{\eta \in \Gamma} \xi_{\eta}  \delta_{\eta} \right)
= \sum_{\eta \in \Gamma} \xi_{\eta} 
\delta_{\eta \gamma^{-1}}.$$

Hence,
$$\Zz \Gamma \subset B(\ell^2 \Gamma)^{\rho(\Gamma)} \deq L\Gamma,$$
where $L\Gamma$ denotes the group von Neumann algebra of $\Gamma$. The algebra $L\Gamma$ carries
a faithful, positive and normal trace state:
$$\tau(a) = \langle a\delta_e,\delta_e \rangle.$$ Given the canonical faithful trace, it makes sense to talk about \emph{the} spectral measure of a 
normal operator.
Indeed, all the spectral projections of a normal operator $a \in L\Gamma$ lie in $L\Gamma$ and applying the trace to the usual projection-valued spectral measure gives an ordinary measure on the complex
plane. Given a finite von Neumann algebra with a specified faithful and normal trace $(M,\tau)$, and a normal operator $a \in M$, we denote the resulting spectral measure by $\mu_a$.
In case of group rings $\Zz \Gamma \subset L \Gamma$, this measure encodes most interesting information about the particular element of $\Zz\Gamma$ and
there are many prominent conjectures about its properties, see \cite{lueck}.
\mn
In \cite{thom2}, we defined the notion of integer operator and showed that as long as $\Gamma$ is a sofic group, all self-adjoint operators in $\Zz\Gamma$ are integer operators. The study of those operators resulted
in various restrictions on the spectral measure. In particular, we showed that
there cannot be any eigenvalues which are not algebraic integers, proving the algebraic eigenvalue conjecture of
J.\ Dodziuk, P.\ Linnell, V.\ Matthai, T.\ Schick and S.\ Yates, see \cite{5authors}. 
Moreover, we gave several restrictions on the operator norm of such integer operators and proved that the continuous
part of the spectral measure of a self-adjoint integer operator of norm $2$ is arcsine distributed. In this article
we will extend and generalize these results.
\mn
The paper is organized as follows: Section $1$ is the Introduction. In Section $2$, we define the concepts of \emph{integer measure} and \emph{integer operator} and study their basic properties. We discuss their occurrence and define the basic invariants, i.e.\ the determinant
and the discrimant, which we use to study those operators. In Section $3$, we use potential theory to obtain restrictions
on the possible spectral measures of integer operators. Our reference for the basic and more advanced results in
one variable potential theory is \cite{pot}. We will show that that there is an abundance of integer
measures, by showing that equilibrium measures are integer measures as soon as there support has logarithmic capacity
bigger or equal than one. Although we refer to results of Fekete-Szeg\"o and Rumely, this section is self-contained. In Section $4$, we relate integer measures to recent work of Petracovici, Petracovici and Zaharescu in \cite{za1}. There, the authors constructed a certain metric space out of the variety of Galois orbits in the
complex plane. We identify this space with a subspace of the Wasserstein space of the complex plane. Moreover,
we proof that the Wasserstein space of a proper metric space is proper.

\section{Integer measures}

\subsection{Definitions and conventions}
\label{intro2}

A polynomial $p \in \Zz[t]$ is said to be \emph{monic} if its leading coefficient is equal to $1$.
Any complex number which can arise as root of a monic integer polynomial is said to be an \emph{algebraic integer}.  
Let $p \in \Zz[t]$ be a monic polynomial with integer coefficients. 
It is well-known, that the roots of $p(t)$ are closed under the action of the 
absolute Galois group $\gal$. They form an orbit of a point if and only if $p(t)$ is an irreducible integer polynomial, i.e.\ does not further factorize into integer polynomials. For background on Galois theory,
we refer to \cite{weil}. For our purposes it suffices to think of the set of all algebraic integers as partitioned into
equivalence classes of those which arise always simultaneously as roots of monic integer polynomials.
These equivalence classes are precisely the orbits of the action of the Galois group.
It is a basic fact of Galois theory that the set of algebraic integers form a sub-ring of $\Cz$.

Given a polynomial $p \in \Cz[t]$ of degree $n$ with roots $\alpha_1,\dots,\alpha_n$, we set:
$$\det(p) = \prod_{\alpha_i \neq 0} \alpha_i, \quad \mbox{and} \quad \disc(p) = \prod_{\alpha_i \neq \alpha_j} (\alpha_i - \alpha_j).$$

It is obvious from their definitions, that if $p$ is an monic 
integer polynomial, then $\det(p)$ and $\disc(p)$ are fixed by the Galois group. It is a another 
basic fact, that all algebraic integers which are fixed by the Galois group are elements of $\Zz$. We conclude that
$|\det(p)| \geq 1$ and $|\disc(p)| \geq 1$ in this case.
\mn
We denote the $C^*$-algebra of bounded complex valued continuous functions on $\Cz$ by $C_b(\Cz)$.
For us, a probability measure $\mu$ on the complex plane $\Cz$ is a linear functional 
$$\int_{\Cz} \cdot \ d\mu(z) \colon C_{b}(\Cz) \to \Cz,$$
which satisfies $\int_{\Cz} f \ d\mu(z) \geq 0$ if $f \geq 0$ and $\int_{\Cz} 1 \ d\mu(z) =1$.
A probability measure is said to be \emph{atomic} if it is a weighted sum of point evaluations.

\begin{definition} Let $p \in \Cz[t]$ be a polynomial. We denote by 
$\Delta(p)$ the atomic probability measure on $\Cz$ which is equidistributed
on the roots of $p(t)$, taking multiplicities into account, i.e.
$$\int_{\Cz} f(z) d(\Delta(p))(z) = \frac1{\deg(p)} \sum_{i=1}^{\deg(p)} f(\alpha_i),$$
where $\alpha_1,\dots,\alpha_{\deg(p)}$ is a list of roots of the polynomial $p$.
\end{definition}

We want to study measures which arise as certain weak limits of measures $\Delta(p)$,
where $p$ is an monic polynomial with integer coefficients.
Recall, a probability measure $\mu$ on $\Cz$ 
is the weak limit of a sequence of probability measures $\mu_n$ on $\Cz$, if
$$\int_{\Cz} f d\mu_n \to \int_{\Cz} f d\mu, \quad \forall f \in C_b(\Cz).$$
\begin{definition}\label{inte}
A probability measure on the complex plane is said to be an \emph{integer measure}, if there exists a sequence
of monic polynomials $p_n \in \Zz[t]$, such that
\begin{enumerate}
\item $\Delta(p_n) \to \mu$ weakly, and
\item The supports of $\Delta(p_n)$ are uniformly bounded, i.e. 
$$\exists \lambda>0 \ \forall n \in \Nz \ \forall \alpha \in \Cz \quad p(\alpha)=0 \Rightarrow |\alpha| \leq \lambda.$$
\end{enumerate}
\end{definition}

In the sequel, we will derive some non-trivial properties of the spectral measures of integer operators, which arise through
the subtle interplay between the action of the absolute Galois group $\gal$ and the metric topology on $\Cz$. However,
there is also a trivial consequence of the fact that complex conjugation is continuous with respect to the metric
topology on $\Cz$. In fact: $\int x^n \ d\mu(x) \in \Rz$, for all $n \in \Nz$ and all integer measures $\mu$. Phrasing it differently, 
and denoting by $\tau\colon \Cz \to \Cz$ complex conjugation, then $\tau_*(\mu) = \mu$, for every integer measure $\mu$.
If a measure satisfies $\tau_*(\mu) = \mu$, we call it \emph{symmetric}. The same term is used for subsets of
the complex plane, which are fixed by the complex conjugation.
\mn
In order to justify the definition, we mention an old result by T.\ Motzkin in \cite{Mo}. 
It says that the second condition in Definition \ref{inte} is essential.

\begin{theorem}
Let $\mu$ be a symmetric probability measure on the complex plane. There exists a sequence $p_n \in \Zz[t]$ of
monic polynomials, such that $\Delta(p_n) \to \mu$ weakly.
\end{theorem}
\begin{proof} Indeed, T.\ Motzkin proved in \cite{Mo} that for all
$\epsilon>0$, all but one of the zeros of a monic integer polynomial can be prescribed to lie 
in an $\epsilon$-neighborhood of arbitrarily given complex numbers, which lie symmetrically in
the complex plane. This implies that every symmetric atomic probability measure can be approximated weakly
by measures of the form $\Delta(p)$ for monic integer polynomials $p \in \Zz[t]$. 
The proof is finished by observing that the set of symmetric atomic probability 
measures is weakly dense in the space
of all symmetric probability measures.
\end{proof}

\subsection{Group rings and integral sub-rings in finite von Neumann algebras}

The following definition differs slightly from Definition $5.1$ in \cite{thom2}. However, all results about integer operators
in \cite{thom2} will be reproved and considerably extended in the sequel.

\begin{definition}
Let $(M,\tau)$ be a finite von Neumann algebra.
A normal operator $a \in M$
is said to be \emph{integer}, if its spectral measures is an integer measures on $\Cz$.
\end{definition}

In many important constructions of finite von Neumann algebras, i.e. as group von Neumann algebras or
as a group measure space construction, there are natural candidates of sub-rings, whose normal elements
should be integer operators. More specifically, we have the following conjecture:

\begin{conjecture} \label{intgroups}
Let $\Gamma$ be a discrete group. 
Then, the subset of normal elements in $\Zz \Gamma \subset L\Gamma$ consist entirely of integer operators.
\end{conjecture}

For residually finite groups, the answer to Conjecture \ref{intgroups}
is positive.
In \cite{thom2}, we showed that for $\Gamma$ sofic, Conjecture \ref{intgroups} has an affirmative answer
for self-adjoint elements in $\Zz \Gamma$.  For a definition of the term \emph{sofic} we refer the reader to \cite{thom2}.
Since there is no group known, which is not sofic, it will be obviously rather hard to construct a counterexample to Conjecture \ref{intgroups}. 

An extension of the results of \cite{thom2} to normal elements could be based upon a better understanding of pairs of 
symmetric integer matrices, whose commutator has a small rank compared to the size of the matrices. More
specifically, a positive answer to the following conjecture about integer matrices would imply the desired result:

\begin{conjecture}
For every $\varepsilon>0$, there exists $\delta>0$, such that the following holds:
Let $A,B \in M_n \Zz$ be symmetric matrices and assume that ${\rm rank}(AB-BA) < \delta \cdot n$ holds.
Then, there exists $k \in \Nz$, and symmetric matrices $A',B' \in M_{nk} \Zz$, such that
$${\rm rank}(A\otimes 1_k-A') + {\rm rank}(B\otimes 1_k -B') < \varepsilon \cdot nk, \quad \mbox{and} \quad A'B'=B'A'.$$ 
\end{conjecture}

The preceding conjecture should be seen as an integral analogue of H.\ Lin's seminal result about almost commuting
self-adjoint complex matrices, see \cite{lin}. There, the corresponding result is proved with the operator norm 
in place of the normalized rank.

The evidence for the following conjecture (which is just the group-measure space analogue of Conjecture \ref{intgroups}) 
is less striking.

\begin{conjecture} \label{intact}
Let $(X,\mu)$ be a standard probability space and $\Gamma$ be a discrete group that acts on $(X,\mu)$
by Borel isomorphisms that preserve the measure.
Then, the normal operators in 
the sub-ring $$L^{\infty}(X,\Zz) \rtimes_{\rm alg} \Gamma \subset L^{\infty}(X,\Cz) \rtimes \Gamma$$ 
are integer operators.
\end{conjecture}

Apart from the amenable case, and pro-finite actions of residually finite groups, there seems to very little
evidence, that Conjecture \ref{intact} ist true. However, this conjecture is in the line of problems like 
Connes embedding problem to which no counterexamples are in sight. 
It would be nice to have a suitable notion of \emph{soficity} 
for Borel actions on standard probability spaces. Although several definitions come to mind immediately,
is not at all clear whether the class of such actions is large or small. Note, that also the Connes embedding
question is essentially open for such von Neumann algebras. An affirmative answer of Question
\ref{intact} would have interesting applications in the theory of $\ell^2$-torsion of discrete groups.

\subsection{Properties of integer measures}
We first recall some properties of integer measures that were obtained in \cite{thom2}. Since the results
are not stated precisely in this form in \cite{thom2}, we will also include a short proof.
For $\beta \in \Cz$ a complex number and $\varepsilon>0$, we denote by $B(\beta,\varepsilon)$
the open disk of radius $\varepsilon$ with centre $\beta$.

\begin{theorem}
Let $\mu$ be an integer measure.
\begin{enumerate}
\item The atoms of $\mu$ are located at algebraic integers and the set of atoms forms an orbit
under the action of the absolute Galois group $\gal$.
\item Galois-conjugate atoms appear with the same size.
\end{enumerate}
\end{theorem}
\begin{proof}
Let $p_n$ be a sequence of monic integer polynomials, such that $\Delta(p_n) \to \mu$ and all zeros of the polynomials
$p_n$ lie within a disk of radius $\lambda$.

Let $\beta \in \Cz$ be an arbitrary complex number. From weak convergence and monotonicity, we get:
\begin{equation} \label{atoms} \mu(\{\beta\}) 
\leq \lim_{\epsilon \to 0} \liminf_{n \to \infty}\Delta(p_n)(B(\beta,\epsilon)). \end{equation}
Assume that $\beta$ is not an algebraic integer. We have to exclude the possibility, that for $\epsilon>0$ arbitrarily small, 
there exist a monic integer polynomials $p$, such that the proportion of zeros within an $\epsilon$-distance of $\beta$ is greater than some $\delta>0$. Clearly, we can restrict to irreducible polynomials.

Assume that such a $\delta>0$ exists. If $p$ is irreducible of degree $n$, then all roots are distinct and we get
\begin{equation} \label{eq5}
1 \leq |\disc(p)| \leq (2\epsilon)^ {\delta n(\delta n-1)} (2\lambda)^{n^2}.\end{equation}
Here, we used that $|\disc(p)| \geq 1$ for monic integer polynomials, which was justified in the beginning of Section
\ref{intro2}.
Taking logarithms in Equation \ref{eq5}, we get:
$$ \delta n (\delta n-1) \leq n^2 \frac{\log 2 \lambda}{|\log 2 \epsilon|},$$
and taking $\epsilon$ sufficiently small we get an upper bound on $n$. Since there are only finitely many
monic integer polynomials of bounded degree which have all their roots in $B(0,\lambda)$, the set of zeros of
such polynomials is discrete and consequently, $\beta$ is an algebraic integer. This contradicts the 
hypothesis that $\delta>0$ exists and proves the first claim.

The argument shows, that if there is an atom at some complex number $\beta$, then it has to be an atom
of $\Delta(p_n)$ for large $n$ and actually 
$$ \mu(\{\beta\}) \leq \liminf_{n \to \infty} \Delta(p_n)(\{\beta\}).$$ The inequality
$$\limsup_{n \to \infty} \Delta(p_n)(\{\beta\}) \leq \mu(\{\beta\})$$ follows from weak convergence
$\Delta(p_n) \to \mu$, and hence \begin{equation} \label{eqapp} \mu(\{\beta\}) = \lim_{n \to \infty} \Delta(p_n)(\{\beta\}).\end{equation} 
This also implies the second claim since the zeros of $p_n$ form 
orbits under the absolute Galois group $\gal$ and the multiplicities are constant along orbits. Thus the proof is finished.
\end{proof}

\begin{proposition} \label{prop1}
Let $\mu$ be an integer measure. The measure $\mu$ is a convex combination of an atomic integer measure and
an integer measure which does not contain atoms.
\end{proposition}

It is conjectured that for a torsion-free group $\Gamma$, the spectral measure of a normal operator in $\Zz\Gamma$
has neither an atomic nor a singular part. The statement about the atomic part is a consequence of Atiyah's conjecture,
see \cite{5authors}. The next theorem is a regularity result for the non-atomic part of the spectral
measure of such operators. Since it applies also to
groups that contain torsion, the conclusion has to be much weaker than the conjecture above.

\begin{theorem}
Let $\mu$ be an integer measure not containing atoms. 
There exists a constant $C$, such that for all $\beta \in \Cz$ and $0<\epsilon< 1$
$$\mu(B(\beta,\epsilon)) \leq C \cdot |\log 2\epsilon\, |^{-\frac12}.$$
\end{theorem}
\begin{proof}
Let $p_n$ be a sequence of monic integer polynomials, such that $\Delta(p_n) \to \mu$ weakly, and
the size of the zeros is bounded by some $\lambda>0$.
We know:
$$\mu(B(\beta,\epsilon)) \leq \liminf_{n \to \infty} \Delta(p_n)(B(\beta,\epsilon)).$$
For simplicity, set $j_{n,\epsilon} = \Delta(p_n)(B(\beta,\epsilon))$.
Each polynomial $p_n$ splits into irreducible factors $$p = q_{1,n}^{k_{1,n}} \cdot \dots \cdot q_{l_n,n}^{k_{l_n,n}}. $$
Let $D_n$ be the degree of $p_n$ and $d_n$ be the minimum degree of the factors of $p_n$. 
Since there are no atoms in $\mu$ and the set of polynomials with bounded degree and all roots within
$B(0,\lambda)$ is finite, the degrees of the factors have to tend to infinity as $n \to \infty$, i.e.
$\lim_{n \to \infty} d_n = \infty$.
It follows, that:
\begin{equation} \label{eq2} 1 \leq \prod_{\alpha_i \neq \alpha_j} |\alpha_i - \alpha_j | \\
\leq (2\epsilon)^{j_{n,\epsilon} D_n^2(j_{n,\epsilon} - 1/d_n)} (2\lambda)^{D_n^2},\end{equation}
where $\alpha_1,\dots,\alpha_{D_n}$ is a list of roots of $p_n$.
In the second inequality above, we have used that the highest power to which an irreducible 
factor can appear is bounded from above by $D_n/d_n$.
Hence, every root can appear at most $D_n/d_n$ times. We conclude from Equation \ref{eq2} that
$$j_{n,\epsilon} D_n^2(j_{n,\epsilon} - 1/d_n)
\leq D_n^2 \frac{\log 2 \lambda}{|\log2\epsilon|},$$
and hence
$$\mu(B(\beta,\epsilon)) \leq \liminf_{n \to \infty} \ j_{n,\epsilon} \leq \lim_{n \to \infty} \left(\frac{\log 2 \lambda}{|\log2\epsilon|} + \frac1{d_n} \right)^{1/2}= \left(\frac{\log 2 \lambda}{|\log2\epsilon|} \right)^{1/2}.$$
This finishes the proof.
\end{proof}

\subsection{Continuity of the determinant and the discriminant}

In this section we investigate the continuity properties of the functional 
$$\log \det(\mu) = \int_{\Cz} \log|z| d\mu(z)$$ on the a suitable space of integer measures.
It turns out that upper semi-continuity can be easily established whereas we find an explicit sequence of 
integer polynomials which shows that $\log \det$ cannot be lower semi-continuous.

\begin{definition}
Let $\mu,\nu$ be probability measures on $\Cz$. We set, for a Borel subset $B \subset  \Cz$:
$$\mu \ast \nu(B) = (\mu \times \nu)(a^{-1}(B)),$$ where $a\colon \Cz^2 \to \Cz$ with $a(\alpha,\beta) = \alpha - \beta$.
\end{definition}

Consider the Banach space $C_{\lambda}=C(\overline{B}(0,\lambda))$. Let $\Lambda_{\lambda} \subset C_{\lambda}'$ be
the weak-$*$-closure of the set of integer measures. 

\begin{proposition}
Convolution defines a separately weak-$*$-continuous map
$$.\ast.\colon\Lambda_{\lambda} \times \Lambda_{\lambda} \to \Lambda_{2\lambda}$$
\end{proposition}
\begin{proof}
The only non-trivial assertion is that the convolution of integer measures is indeed an integer measure.
Let $p,q \in \Zz[t]$ be monic polynomials with roots $\alpha_1,\dots,\alpha_n$ and $\beta_1,\dots,\beta_k$.
Clearly, the set $\{\alpha_i + \beta_l \mid 1 \leq i \leq n, 1 \leq l \leq k\}$ is invariant under the action of $\gal$.
This implies that this set is zero-set of a monic integer polynomial. The result extends by continuity.
\end{proof}

\begin{lemma}
Let $\lambda \geq 1$. The map 
$$\log \det \colon \Lambda_{\lambda} \to \Rz, \quad \log \det(\mu) = \int_{z \neq 0} \log |z|\ d\mu(z).$$
is well-defined and weakly upper semi-continuous, i.e.
$$\int_{z \neq 0} \log |z|\ d\mu(z) \geq \limsup_{n \to \infty} \int_{z \neq 0} \log |z|\ d\mu_n(z),$$
for every weakly convergence sequence of integer measures.
\end{lemma}
\begin{proof}
Upper semi-continuity in the weak-$*$ topology follows from a variant of Fatou's Lemma:

\begin{lemma}[Fatou] Let $\mu_k \to \mu$ weakly and let $f\colon \Cz \to \Rz$ be a measurable and  positive function.
If $$\liminf_{k \to \infty} \int_{\Cz} f(x)\ d\mu_k(x) < \infty,$$ then the right side of the inequality exists and
$$\liminf_{k \to \infty} \int_{\Cz} f(x)\ d\mu_k(x) \geq \int_{\Cz} f(x)\ d\mu(x).$$
\end{lemma}
\begin{proof}
\begin{eqnarray} \nonumber
\int_{\Cz} f(x) d\mu(x) &=& \sum_{n=0}^{\infty} \int_{f^{-1}([n,n+1))} f(x) \ d\mu(x) \\\nonumber
&=&\sum_{n=0}^{\infty} \lim_{k \to \infty }\int_{f^{-1}([n,n+1))} f(x) \ d\mu_k \\\nonumber
&\leq & \liminf_{k \to \infty } \sum_{n=0}^{\infty} \int_{f^{-1}([n,n+1))} f(x)\ d\mu_k \\ 
&=& \liminf_{k \to \infty } \int_{\Cz} f(x)\ d\mu_k \\ \nonumber
\end{eqnarray}
Here, we used the more familiar variant of Fatou's Lemma to get the inequality.\end{proof}

Specifying the preceding lemma to $f(z)= \max\{\log \lambda - \log|z|,0\}$, this implies 
that $\log \det$ is well-defined (i.e.\ bounded on $\Lambda_{\lambda}$) and upper semi-continuous on $\Lambda_{\lambda}$.
For this, note that always $\log \det(\Delta(p)) \geq 0$ since $p$ is a monic integer polynomial and hence: 
$$ \liminf_{k \to \infty} \int_{z \neq 0} \max\{\log \lambda - \log|z|,0\}\ d (\Delta(p_k))(z) \leq \pi \lambda^2 \cdot \log \lambda.$$ 

\end{proof}

\begin{corollary} \label{disccont}
The map 
$$\log \disc \colon \Lambda_{\lambda} \times \Lambda_{\lambda} \to \Rz, 
\quad \log \disc(\mu,\nu) = \int_{z \neq w} \log |z-w|\ d\mu(z) d\nu(w).$$
is well-defined upper semi-continuous from the weak-$*$-topology.
\end{corollary}
\begin{proof}
We write $\log \disc(\mu,\nu) = \log \det(\mu \ast {\nu}),$ using the factorization:
$$\Lambda_{\lambda} \times \Lambda_{\lambda} \to \Lambda_{2 \lambda} \stackrel{\log \det}{\to} \Rz.$$
\end{proof}

For simplicity, we set $D(\mu) = \log \disc(\mu,\mu).$ The following example shows that $\log \det$ is in general not continuous. I owe this example to K.\ Ramsay.
\begin{example}
Let $0 \neq \lambda \in 2\Nz$. Consider 
$p_n(t) = t^n + 2^{-1} \lambda^{n-1} t +2$. For $|t|= \lambda$ and sufficiently large $n$, we have
$$|t^n+2| \geq \lambda^n - 2 > 2^{-1} \lambda^n = |2^{-1} \lambda^{n-1} t|.$$
Hence, by Rouch\'{e}'s Theorem, $p_n(t)$ has all its roots within $|t|= \lambda$, i.e.\ $\Delta(p_n) \in \Lambda_{\lambda}$.
For $|t|= \lambda^{1-n/2} $ and $n$ sufficiently large, we have
$$|t^n + 2| \leq 3 < |2^{-1} \lambda^{n/2}| = |2^{-1} \lambda^{n-1} t|.$$
Again, using Rouch\'{e}'s Theorem, $p_n$ will have precisely one root $t_0$ inside $B(0,\lambda^{1-n/2})\subset B(0,1)$.
Moreover, by a similar argument, this is the only root in $B(0,1)$. We compute
$$\liminf_{n \to \infty} \frac{|\log t_0|}n \geq \lim_{n \to \infty} \frac{\log \lambda^{n/2-1}}n = \frac{\log{\lambda}}2.$$
This implies, that
$$\limsup_{n \to \infty} \int_{|z| \leq 1} \log|z| \ d\Delta(p_n) \leq - \frac{\log{\lambda}}2,$$
whereas $\int_{|z| \leq 1} \log|z| \ d\mu =0$, where $\mu$ denotes some limiting point of the sequence $\Delta(p_n)$.
Indeed, note that $\mu(B(0,1)) \leq \limsup_{n \to \infty} \Delta(p_n)(B(0,1)) = 0$. Clearly, 
$$\lim_{n \to \infty} \int_{|z| \geq 1} \log|z| \ d\Delta(p_n) = \int_{|z| \geq 1} \log|z| \ d\mu,$$
and hence
$$\limsup_{n \to \infty} \int_{z \neq 0} \log|z| \ d\Delta(p_n) < \int_{z \neq 0} \log|z| \ d\mu.$$
This shows that $\log \det$ is not continuous.
\end{example}

\section{Potential theory and applications}

\subsection{Free entropy, logarithmic energy, capacity and discriminants}

Let $\mu$ be a probability measure. We set:
\begin{eqnarray*}I(\mu) &=& \iint_{E\times E} - \log |z-w| d\mu(z) d\mu(w)\\
&\deq&\lim_{n \to \infty}\iint_{E\times E} \min\{n, - \log |z-w|\} d\mu(z) d\mu(w)
\end{eqnarray*}
This quantity is called the \emph{logarithmic energy} of the measure $\mu$. If $\mu$ is the spectral measure of a self-adjoint
operator $a \in M$, D.\ Voiculescu defined the free entropy $\chi(a)$, which is closely 
related to the logarithmic energy of the measure $\mu$. In fact, we have
$$\chi(a) = \frac34 + \frac{\log(2\pi)}2 - I(\mu_a).$$

The following proposition is the basic observation, which gives a first restriction on the possible values
of $I(\mu)$ and $\chi(a)$, if $a \in M$ is an integer operator.

\begin{proposition} \label{comp}
Let $p_n \in \Zz[t]$ be a sequence of monic polynomials having all their roots in a disk $B(0,\lambda)$ for
some $\lambda>0$. Assume that $\Delta(p_n) \to \mu$ weakly. If $\mu$ has no atoms, then
$$I(\mu) + \limsup_{n \to \infty} D(\Delta(p_n)) \leq 0.$$
\end{proposition}
\begin{proof}
If $\mu$ has no atoms, then $I(\mu) = - D(\mu)$. Since the discriminant is upper semi-continuous by Corollary
\ref{disccont}, $D(\mu) \geq \limsup_{n \to \infty} D(\Delta(p_n))$. This implies the claim.
\end{proof}

\begin{theorem} \label{inequality}
Let $(M,\tau)$ be a finite von Neumann algebra and let $a \in M$ be self-adjoint integer operator without atoms. Then,
$$I(\mu_a) \leq 0, \quad \mbox{and} \quad \chi(a) \geq \frac34 + \frac{\log(2\pi)}2.$$
\end{theorem}
\begin{proof}
This result follows from Proposition \ref{comp}, since $\mu_a$ is an integer measure and
$$\exp D(\Delta(p_n)) =  \prod_{\alpha_i \neq \alpha_j} |\alpha_i - \alpha_j|  \geq 1,$$ for all monic polynomials $p_n \in \Zz[t]$. Here, $\alpha_1,\dots,\alpha_n$ is a list of roots of $p_n$.
\end{proof}

Let $E\subset \Cz$ be a compact sub-set of the complex numbers. Recall, the \emph{logarithmic capacity} $\cap(E)$ is defined as
$$-\log(\cap(E)) = \inf\ \{I(\mu)\mid \mu \mbox{ is supported on $E$}\}.$$
If $\cap(E) >0$, then there 
exists a \emph{unique} measure $\mu_E$ on $E$, which is called the \emph{equilibrium measure} $\mu_E$ on $E$,
such that $-\log(\cap(E)) = I(\mu_E)$.

\begin{example}
The equilibrium measure on the unit circle is just the Haar measure, whereas the equilibrium measure on the interval
$[-2,2]$ is the arcsine law, which has density with respect to the Lebesgue measure:
$$\frac{d\mu}{d \lambda} = \frac1{\pi \sqrt{4-x^2}}.$$
In both cases, the logarithmic capacity is equal to one.
\end{example}

\begin{example}
Let $p \in \Cz[t]$ be a monic polynomial and consider the lemniscate
$$E= \{z \in \Cz \mid |p(z)| = R\}.$$
Then, for the logarithmic capacity, we get: $\gamma(E) = R^{1/\deg(p)}.$ This is best understood through
the 
explicit Green's function and some well-known results from potential theory, which we recall in Section \ref{green}. For more details, see also \cite{pot}.
\end{example}

\subsection{The work of Fekete and Szeg\"o}

We first want to recall a main result in the work of M.\ Fekete and G.\ Szeg\"o, which relates the logarithmic capacity
of a subset of the complex plane with the number of orbits of algebraic integers, that can be found near those sets.
A set of algebraic integers is said to be \emph{complete}, if it is closed under the action of the absolute Galois group.

\begin{theorem}[Fekete-Szeg\"o] \label{feke0}
Let $E \subset \Cz$ be a compact set which is symmetric with respect to complex conjugation.
\begin{enumerate}
\item If the logarithmic capacity satisfies $\cap(E)<1$, then there exists an open neighborhood $U$ of $E$, such that only
a finite number of complete sets of algebraic integers are contained in $U$.
\item If the logarithmic capacity satisfies $\cap(E)\geq 1$, then every open neighborhood $U$ of $E$ contains an
infinite number of complete sets of algebraic integers.
\end{enumerate}
\end{theorem}

Observe, that this result has a very drastic consequence on the spectral measure, supported on sets of small logarithmic capacity.

\begin{proposition} \label{feke1}
Let $\mu \in \Lambda_{\lambda}$ be an integer measure. 
If the capacity of $\supp(\mu)$ is strictly less than $1$, then
the measure $\mu$ is completely atomic.
\end{proposition}
\begin{proof}
If the measure $\mu$ is not completely atomic, then, by Proposition \ref{prop1}, 
its continuous part $\mu_c$ is an integer measure, which 
is supported on $\supp(\mu_c) \subset \supp(\mu)$. It follows from Theorem \ref{inequality} 
that $I(\mu) \leq 0$ and hence $\cap(\supp(\mu)) \geq \cap(\supp(\mu_c))\geq 1.$ This finishes the proof.
\end{proof}

Theorem \ref{integ} shows the abundance of spectral measures. For the proof, we needed 
the following theorem, which 
is a slight but enlightening modification of a result of Fekete and Szeg\"o. Since its proof can be taken
almost verbatim from their seminal article \cite{feksze}, we attribute this result to them. 
Indeed, it is most likely that this variant is well-known to experts.

\begin{theorem}[Fekete-Szeg\"o] \label{key}
Let $p(t) \in \Cz[t]$ be a monic polynomial and the lemniscate domain
$$E=\{z \in \Cz \mid |p(z)| = R^{\deg(p)}\}.$$

For every $Q<R$, there exists an monic integer polynomial $q(t) \in \Zz[t]$, such
that the set $$\{z \in \Cz \mid |q(z)| = Q^{\deg(q)}\}$$ is enclosed by $E$.
\end{theorem}

The main consequence in the language of integer measures is now the following:

\begin{theorem} \label{integ}
Let $E \subset \Cz$ be a symmetric sub-set and of the complex plane and 
denote by $\mu_E$ the equilibrium measure. If $\cap(E) \geq 1$,
then $\mu_E$ is an integer measure.
\end{theorem}
\begin{proof}The proof is a reduction procedure, whose key step is Theorem \ref{key}.

First of all, by Theorem $I$ in \cite{feksze}. There is a sequence of lemniscate domains 
$E_n = \{z \mid |p_n(z)|=R^{\deg(p_n)} \}$, 
such that 
\begin{enumerate}
\item $E_{n+1}$ encloses $E_n$,
\item as $n \to \infty$, the logarithmic capacity of the lemniscate 
domains $\gamma(E_n) = R$ converges to $\gamma(E)$, and
\item the curves approach $E$ in the Hausdorff distance.
\end{enumerate}
We consider the sequence of probability measures $\mu_{E_n}$, that arise as equilibrium measures on the sets $E_n$. 
This sequence, by the Theorem 
of Banach-Alaoglu, has a weak-$*$-convergent subsequence, which converges to a measure $\mu$. Clearly, the measure $\mu$ is supported
on $E$ and $$I(\mu) \geq \limsup_{n \to \infty} I(\mu_{E_n}) = -\log \gamma(E).$$ 
Hence, $I(\mu) = - \log \gamma(E)$ and
$\mu$ has to be the equilibrium measure on $E$.
To finish the proof, we need to show that an equilibrium measure on a lemniscate domain is an integer measure.

We now employ Theorem \ref{key}. Recall, for $Q < R$ be arbitrary, Theorem \ref{key} yields a monic
integer polynomial $q \in \Zz[t]$, such that $E_{R,p}=\{z \mid |p_n(z)|=R^{\deg(p_n)} \}$
encloses $E_{Q,q} = \{z \mid |q(z)|=Q^{\deg(q)} \}$.
 
The set of real numbers $Z=\{m^{1/k} \mid m,k \in \Nz\}$ is dense in $\Rz_{\geq 0}$. Hence, we find
$Q_n \in Z$, such that $Q_n < R$ and $|R - Q_n|< 1/n$. Denote by $q_n$ the monic integer polynomial
that corresponds to $Q_n$. Observe, 
the equilibrium measure on $E_{Q_n,q_n}$ is an integer measure. 
Indeed, if $Q_n = {m_n}^{\deg(q_n)/k_n}$, then $$E_{Q_n,q_n} = \{ z \mid |q(z)|^{k_n} = m_n^{\deg(q_n)} \}$$
supports all the roots of the monic and integer polynomials 
$m_l(t) = q(t)^{k_n l} - m_n^ {l \deg(q_n)}$. Moreover, it follows from well-known results in \cite{feksze}, that 
$\Delta(m_l) \to \mu_{E_{Q_n,q_n}}$ weakly.

Again, applying the Theorem of Banach-Alaoglu, the sequence of probability measure $\mu_{E_{Q_n,q_n}}$
has a weak limit point $\mu'$, which is an integer measure. Denote the support of $\mu'$ by $E'$.

We conclude from the construction
that the set $E'$, which is enclosed by $$\left\{z \mid |p_n(z)|=R^{\deg(p_n)} \right\}.$$
Moreover, $$I(\mu') \geq 
\limsup_{n \to \infty} \left( - \log \gamma\left(\left\{z \mid |q_n(z)|=Q_n^{\deg(q_n)} \right\}\right) \right)
= \lim_{n \to \infty} Q_n = R.$$ 
Hence, we have constructed an integer measure $\mu'$ whose support is enclosed by $E$, and
such that $I(\mu') = R=-\log \gamma(E)$. Since the logarithmic capacity of $E$ and the set enclosed by $E$ agree, 
it follows that $\mu'$ is the equilibrium measure on $E=\{z \mid |p_n(z)|=R^{\deg(p_n)} \}$.
This finishes the proof.
\end{proof}

\subsection{Green function's and Rumely's Theorem} \label{green}

We seek a partial converse for Theorem \ref{integ}.
Let $E \subset \Cz$ be a compact subset and assume that $E^c$ is connected. We can define the Green function
$G_E\colon \Cz \to \Rz$ which is uniquely determined by the following properties:

\begin{enumerate}
\item[(i)] it is continuous and non-negative on $\Cz$,
\item[(ii)] it vanishes on $E$ and is harmonic on $E^c$,
\item[(iii)] $G_E(z) - \log|z |$ is bounded near $\infty$.
\end{enumerate}

It is well-known that $\lim_{z \to \infty} G_E(z) - \log|z| = - \log \cap(E)$, see \cite{pot}.

\begin{example}
Let $p(t) \in \Cz[t]$ be a polynomial of degree $n$. The set $$E=\{z \in \Cz \mid |p(z)| \leq R\}$$ is compact and
$$G_E(z) = \begin{cases} 0 & z \in E \\ n^ {-1} ( \log |p(z)| - \log R ) & z \in E^ c \end{cases},$$
and one concludes that $\cap(E) = R^ {1/n}$.
\end{example}

Let us now define the \emph{logarithmic Weil height} with respect to $E$. Let $p$ be a polynomial with roots $\alpha_1,\dots,\alpha_n$.
We set: $$h_E(p) = \frac1n \sum_{i=1}^ n G_E(\alpha_i) = \int_{\Cz} G_E(z) d\Delta(p).$$
\begin{example}
If $E = B(0,1)$, then $G_E(z) = \max\{0,\log |z| \}$ and $h_E(p) = \frac1n \log M(p)$ is the renormalized Mahler
measure of the polynomial $p$.
\end{example}
The main result of R.\ Rumely in \cite{rumely}, which generalizes a seminal equidistribution result of Y.\ Bilu, see \cite{bilu}, is the following:

\begin{theorem}[Rumely] \label{rume}
Let $E \subset \Cz$ be a compact set which is symmetric with respect to complex conjugation. Assume that
the logarithmic capacity satisfies $\cap(E)=1$.
Let $\alpha_n$ be a sequence of algebraic integers with $\deg(\alpha_n) \to \infty$ .
Denote by $p_n$ the minimal polynomial of $\alpha_n$. If $h_E(p_n)\to 0$, then the 
measures $\Delta(p_n)$ converge weakly to the equilibrium distribution $\mu_E$ on $E$.
\end{theorem}
\begin{proof}
Obviously, $h_E(p_n) \to 0$ is equivalent to $\supp(\mu) \subset E$. Since $I(\mu) \leq 0$ by Proposition \ref{comp}, 
and $\cap(E) =1$, we conclude $I(\mu)=0$. Thus, $\mu=\mu_E$ by uniqueness of the equilibrium measure.
\end{proof}

\begin{corollary}
Let $\mu$ be an integer measure without atoms. If $\cap(\supp(\mu)) =1$, then $\mu=\mu_E$.
\end{corollary}
\begin{proof} Let us first assume that $\mu$ is irreducible.
Any irreducible integer measure is the weak limit of measures $\Delta(p_n)$, with $p_n \in \Zz[t]$, 
irreducible and monic. If the limiting measure is supported on a compact set $E \subset \Cz$, then
$$h_E(p_n) = \int_{\Cz} G_E(z) d\Delta(p) \to \int_{\Cz} G_E(z) d\mu =0.$$
Hence, Rumely's Theorem applies and $\mu = \mu_E$. Now, since any integer 
measure on $E$ lies in the closed convex hull of irreducible integer measures on $E$,
the result immediately extends to non-irreducible measures.
\end{proof}
The next corollary is a generalization of results, that have been obtained in \cite{thom2} in case of
the unit circle and the interval $[-2,2]$. Note that the proof is simplified.
\begin{corollary}
Let $\mu$ be an integer measure without atoms. If $\cap(\supp(\mu)) =1$, then there exists a sequence
$\lambda_0,\lambda_1,\lambda_2, \dots$ of non-negative real numbers, and a sequence of monic irreducible 
integer polynomials $p_1,p_2,\dots$, such that 
$$\sum_{i=0}^{\infty} \lambda_i =1,\quad \mbox{and} \quad
\mu=\lambda_0 \mu_E + \sum_{i=1}^{\infty} \lambda_i \Delta(p_i).$$
\end{corollary}

\subsection{Decomposition of integer measures}

\begin{theorem} \label{decomp}
We denote by $P_\lambda$ the sub-set of probability measures in $\Lambda_{\lambda}$ and
$X_{\lambda}$ its set of extreme points.
Let $\lambda\geq 1$.
\begin{enumerate}
\item A measure in $X_{\lambda}$ is completely atomic, if and only if it equal to $\Delta(p)$ with $p \in \Zz[t]$,
an irreducible monic polynomial with all roots in $B(0,\lambda)$. 
Otherwise, it does not contain any atoms.
\item The set $X_{\lambda}$ is contained  in the weak closure of
$$\{\Delta(p) \mid p \in \Zz[t] \ \mbox{irreducible and monic}\}.$$
\item For every integer measure $\mu \in P_{\lambda}$ there exists a unique probability
measure $\nu \in P(X_{\lambda})$, such that
$$\mu = \int_{X_C} \mu_x \ d \nu(x).$$
\end{enumerate}
\end{theorem}

It remains to characterize the extreme points. One natural class of measures that comes to mind, is
the class of equilibrium measures. Let us formulate the following somewhat speculative question:

\begin{question} \label{capint}
Let $\mu \in X_{\lambda}$ be an integer measure and let $E=\supp(\mu)$. 
Is it always true that $\mu = \mu_E$?
\end{question}

\begin{remark}
Note that Theorem \ref{rume} solves Question \ref{capint} affirmatively 
if the capacity of $\supp(\mu)$ equals one. 
Moreover, the conjecture is obviously true whenever the capacity is smaller that one.
However, these cases are very special. Also in view of concavity of $\mu \to I(\mu)$, rather
than convexity, a positive solution seems unlikely. In view of Theorem \ref{decomp}, an 
affirmative answer would be a striking result. In fact,
it would reveal the nature of integer measures in a very nice and understandable way.
\end{remark}

\section{The metric space of integer measures}
\subsection{The work of Petracovici, Petracovici and Zaharescu}

In \cite{za1}, B.\ Petracovici, L.\ Petracovici and A.\ Zaharescu introduced the notion of \textit{unitary positive 
divisors} on the complex plane and defined a metric on the set of those divisors. As one easily sees, a positive unitary divisor
is nothing but an atomic probability measure, where the distance function agrees with the well-known Wasserstein metric on the space of probability measures. Moreover, if the underlying metric space is
bounded, it is well-known that convergence in the Wasserstein metric is equivalent to weak convergence,
see \cite{gs} for a nice survey on the relations between the many ways to metrize spaces of measures. 

We see that every integral measure is a point in the completion of the metric space 
of Galois invariant positive unitary divisors, which was studied in \cite{za1}. 
In \cite{za1}, arbitrary (i.e. not necessarily integer) orbits of the Galois group where studied and some interesting
number theoretic constructions found a nice interpretetion in terms of the Wasserstein metric.

In order to study the space metric space which was introduced in \cite{za1} more directly, it is of importance to decide, whether the induced topology is locally compact, i.e. whether the metric is proper. The following theorem subsection settles this question in full generality. 
\subsection{Spaces of measures and properness}

This subsection is independent of the rest of the paper. The result must have appeared somewhere in the literature. However, since we where unable to locate a reference, we include a short proof.

\mn

Let $(X,d)$ be a metric space. We consider the following space
of Borel probability measures on $X$:
$$P(X) = \left\{\mu \mid \int_X d(x_0,x) d\mu(x) < \infty\right\},$$
where $x_0$ is an arbitrary point in $X$. Using the triangle inequality, one easily sees that the definition
does not depend on the choice of the point $x_0 \in X$. We endow $P(X)$ with the Wasserstein metric $d_W$.
For the definition of the Wasserstein metric, recall:
$$d_W(\mu_1,\mu_2) = \inf\left\{ \int_{X \times X} d(x,y) d\nu(x,y)\mid (\pi_1)_*(\nu) = \mu_1, (\pi_2)_*(\nu) = \mu_2\right\},$$
where $\nu$ is supposed to be a probability measure on $X \times X$.

\begin{theorem}
Let $(X,d)$ be a proper, separable and complete metric space. 
The metric space $(P(X),d_W)$ is proper, separable and complete.
\end{theorem}
\begin{proof}
It is well-known that $(P(X),d_W)$ is separable and complete. Let us show that it is indeed proper.

Consider $B_{1} = \left\{\mu\mid \int_X d(x_0,x) d\mu(x) \leq 1 \right\}$
and let $\mu_n$ be a sequence of measures in $B_1$. We want to show that there is a 
convergent sub-sequence $\mu_{n_l}$.
We write:
$$\mu_n = \sum_{i=1}^{\infty} \lambda_{n,k} \mu_{n,k},$$
where $\mu_{n,k}$ is supported in $\overline{B}(x_0,k+1) \setminus B(x_0,k)$.
On bounded domains, the Wasserstein metric coincides with the weak-$*$ topology. Hence,
for each $k$, the sequences $\mu_{n,k}$ have convergent subsequences. Similarly,
the sequences $\lambda_{n,k}$ have convergent sub-sequences.
Using a diagonalization procedure, there
exists an increasing sequence of positive integers $n_l$, such that 
$d_W(\mu_{n_l,k},\nu_k) \leq \frac1{l}$ and 
$|\lambda_{n_l,k} - \lambda_k| \leq \frac1{l \cdot 2^k} $.
The numbers $\lambda_{n,k}$ satisfy:
$$\sum_{j \geq k}\lambda_{n,j} \leq \frac1k \quad \mbox{and} \quad \sum_{k} \lambda_{n,k} =1.$$
We conclude that $\sum_{j=1}^k \lambda_j \geq 1-1/k$. Also, by Fatou's Lemma
$$\sum_{k} \lambda_k \leq \liminf_{l \to \infty} \left(\sum_{k} \lambda_{n_l,k}\right) =1,$$
and hence, $\sum_{k} \lambda_k =1$. We
conclude that $\nu=\sum_{k} \lambda_k \nu_{k}$ defines a probability measure on $X$. 
Moreover,
\begin{eqnarray*}d_W(\mu_{n_l},\nu) &\leq& 
d_W\left(\mu_{n_l},\sum_{k} \lambda_k \mu_{n_l,k}\right) + 
d_W\left(\sum_{k} \lambda_k \mu_{n_l,k},\nu\right)\\
&\leq& \sum_{k} |\lambda_{n_l,k} - \lambda_{k} | \int_X d(x_0,x) d\mu_{n_l,k}+ \sum_{k} \lambda_k d_W(\mu_{n_l,k},\mu_k)\\
&\leq& \sum_k \frac1{l \cdot 2^k} + \sum_k \frac{\lambda_k}l = \frac2l\end{eqnarray*}
To estimate the first summand in the second inequality, we have used that
$$d_W(\mu,\nu) = \sup\left\{ \left|\int_X f(x) d\mu(x) - \int_X f(x) d\nu(x) \right| \mid f\colon X \to \Rz 
\mbox{ contractive}\right\}.$$
This is the Theorem of Kantorovich and Rubinstein, see \cite{kanrub}. In our special case, we get
\begin{eqnarray*}
d_W\left(\sum_{k} \lambda_{n_l,k} \mu_{n_l,k},\sum_{k} \lambda_k \mu_{n_l,k}\right) &\leq& 
\sum_{k}  |\lambda_{n_l,k} - \lambda_{k} | \int_X d(x_0,x) d\mu_{n_l,k} 
\end{eqnarray*}
This proves that $B_1$ is compact. A similar argument applies to $B_{\lambda}$. Since each bounded set is contained
in $B_{\lambda}$ for some $\lambda \in \Rz$, this finishes the proof.
\end{proof}


\begin{bibdiv}
\begin{biblist}

\bib{bilu}{article}{
   author={Bilu, Y.},
   title={Limit distribution of small points on algebraic tori},
   journal={Duke Math. J.},
   volume={89},
   date={1997},
   number={3},
   pages={465--476},
   issn={0012-7094},
}

\bib{5authors}{article}{
   author={Dodziuk, J.},
   author={Linnell, P.},
   author={Mathai, V.},
   author={Schick, T.},
   author={Yates, S.},
   title={Approximating $L\sp 2$-invariants and the Atiyah conjecture},
   note={Dedicated to the memory of J\"urgen K. Moser},
   journal={Comm. Pure Appl. Math.},
   volume={56},
   date={2003},
   number={7},
   pages={839--873},
   issn={0010-3640},
}

\bib{fek}{article}{
   author={Fekete, M.},
   title={\"Uber die Verteilung der Wurzeln bei gewissen algebraischen
   Gleichungen mit ganzzahligen Koeffizienten},
   journal={Math. Z.},
   volume={17},
   date={1923},
   number={1},
   pages={228--249},
   issn={0025-5874},
}

\bib{feksze}{article}{
   author={Fekete, M.},
   author={Szeg{\"o}, G.},
   title={On algebraic equations with integral coefficients whose roots
   belong to a given point set},
   journal={Math. Z.},
   volume={63},
   date={1955},
   pages={158--172},
   issn={0025-5874},
}


\bib{gs}{article}{
   author={Gibbs, A.},
   author={Su, F.},
   title={On choosing and bounding probability metrics},
   journal={International Statistical Review},
   volume={70},
   date={2002},
   pages={419--435}
}

\bib{kanrub}{article}{
   author={Kantorovi{\v{c}}, L. V.},
   author={Rubin{\v{s}}te{\u\i}n, G. {\v{S}}.},
   title={On a space of completely additive functions},
   language={Russian, with English summary},
   journal={Vestnik Leningrad. Univ.},
   volume={13},
   date={1958},
   number={7},
   pages={52--59},
   issn={0146-924x},
 }

\bib{lin}{article}{
   author={Lin, H.},
   title={Almost commuting selfadjoint matrices and applications},
   conference={
      title={Operator algebras and their applications (Waterloo, ON,
      1994/1995)},
   },
   book={
      series={Fields Inst. Commun.},
      volume={13},
      publisher={Amer. Math. Soc.},
      place={Providence, RI},
   },
   date={1997},
   pages={193--233},
}

\bib{lueck}{book}{
   author={L{\"u}ck, W.},
   title={$L\sp 2$-invariants: theory and applications to geometry and
   $K$-theory},
   series={Ergebnisse der Mathematik und ihrer Grenzgebiete. 3. Folge. A
   Series of Modern Surveys in Mathematics},
   volume={44},
   publisher={Springer-Verlag},
   place={Berlin},
   date={2002},
   pages={xvi+595},
   isbn={3-540-43566-2},
 }

\bib{Mo}{article}{
   author={Motzkin, Th.},
   title={From among $n$ conjugate algebraic integers, $n-1$ can be
   approximately given},
   journal={Bull. Amer. Math. Soc.},
   volume={53},
   date={1947},
   pages={156--162},
}

\bib{za1}{article}{
   author={Petracovici, B.},
   author={Petracovici, L.},
   author={Zaharescu, A.},
   title={A new distance between Galois orbits over a number field},
   journal={Math. Sci. Res. J.},
   volume={8},
   date={2004},
   number={1},
   pages={1--15},
   issn={1537-5978},
}

\bib{pot}{book}{
   author={Pommerenke, Chr.},
   title={Univalent functions},
   note={With a chapter on quadratic differentials by Gerd Jensen;
   Studia Mathematica/Mathematische Lehrb\"ucher, Band XXV},
   publisher={Vandenhoeck \& Ruprecht},
   place={G\"ottingen},
   date={1975},
   pages={376},
}

\bib{rumely}{article}{
   author={Rumely, R.},
   title={On Bilu's equidistribution theorem},
   conference={
      title={Spectral problems in geometry and arithmetic (Iowa City, IA,
      1997)},
   },
   book={
      series={Contemp. Math.},
      volume={237},
      publisher={Amer. Math. Soc.},
      place={Providence, RI},
   },
   date={1999},
   pages={159--166},
}

\bib{sze}{article}{
   author={Szeg{\"o}, G.},
   title={Bemerkungen zu einer Arbeit von Herrn M. Fekete: \"Uber die
   Verteilung der Wurzeln bei gewissen algebraischen Gleichungen mit
   ganzzahligen Koeffizienten},
   journal={Math. Z.},
   volume={21},
   date={1924},
   number={1},
   pages={203--208},
}

\bib{thom2}{article}{
   author={Thom, A.},
   title={Sofic groups and diophantine approximation},
   year={2008}
   journal={Comm.\ Pure Appl. Math.},
   status={to appear},
}

\bib{weil}{book}{
   author={Weil, Andr{\'e}},
   title={Basic number theory},
   series={Classics in Mathematics},
   note={Reprint of the second (1973) edition},
   publisher={Springer-Verlag},
   place={Berlin},
   date={1995},
   pages={xviii+315},
   isbn={3-540-58655-5},
}

\end{biblist}
\end{bibdiv}
\end{document}